\documentclass{article}

\usepackage[a4paper,bindingoffset=0.2in,%
left=1in,right=1in,top=1in,bottom=1in,%
footskip=.25in]{geometry}

\usepackage{amssymb}
\usepackage{amsmath}
\usepackage{amsthm}
\newtheorem{thm}{Theorem}
\newtheorem{lem}[thm]{Lemma}
\newtheorem{prop}[thm]{Proposition}
\theoremstyle{definition}
\newtheorem{defn}[thm]{Definition}
\newtheorem{ex}{Example}

\theoremstyle{remark}
\newtheorem{rem}[thm]{Remark}

\newtheorem{cor}[thm]{Corollary}

\newcommand{\RR}{\mathbb R}
\newcommand{\QQ}{\mathbb Q}
\newcommand{\ZZ}{\mathbb Z}

\renewcommand{\O}{\mathcal O}

\newcommand{\be}{\begin{equation}}
\newcommand{\ee}{\end{equation}}

\newcommand{\tensor}{\otimes}
\newcommand{\eps}{\varepsilon}
\newcommand{\wG}{\widetilde{G}}

\DeclareMathOperator{\Aut}{\mathrm{Aut}}

\DeclareMathOperator{\diag}{\mathrm{diag}}

\numberwithin{thm}{section}

\begin{document}

\author{Assaf Goldberger and Ilias Kotsireas}





\title{Formal Orthogonal Pairs via Monomial Representations and Cohomology}

\maketitle

\abstract{A Formal Orthogonal Pair is a pair $(A,B)$ of symbolic rectangular matrices such that $AB^T=0$. It can be applied for the construction of Hadamard and Weighing matrices. In this paper we introduce a systematic way for constructing such pairs. Our method involves Representation Theory and Group Cohomology. The orthogonality property is a consequence of non-vanishing maps between certain cohomology groups. This construction has strong connections to the theory of Association Schemes and (weighted) Coherent Configurations. Our techniques are also capable for producing (anti-) amicable pairs. A handful of examples are given.}






\section{Introduction}
A (classical) \emph{Hadamard Matrix} of order $n$ is a $n\times n$ matrix $H$ over $\{-1,1\}$ such that $HH^T=nI_n$. A \emph{Weighing Matrix} of order $n$ and weight $k$ is a matrix $W$ over $\{-1,0,1\}$ such that $WW^T=kI_n$. We shall say that $H$ is a $H(n)$ and $W$ is a $W(n,k)$. A well-known necessary condition for the existence of Hadamard matrices (excluding the trivial cases $n = 1, 2$) is that $ n \equiv 0 \,\, (mod \,\,4)$. The sufficiency of the condition $n \equiv 0 \,\, (mod \,\, 4)$ is the famous Hadamard Conjecture, which states that an $H(n)$ should exist for every $n$ divisible by $4$. N. J. A. Sloane maintains an on-line database of Hadamard matrices of various orders \cite{NJAS:HM:DB}. The Magma computer algebra system \cite{Magma:1997} features two databases containing Hadamard and skew-Hadamard matrices of various orders.\\

There are hundreds of constructions for Hadamard and Weighing Matrices. One prolific method is to employ \emph{Orthogonal Designs}, see \cite{Geramita:Seberry:LNPAM:1979}, \cite{Seberry:ODs:2017}. This is a symbolic $n\times n$ matrix $O$ with commuting indeterminate entries taken from a set $\{0,t_1,\ldots,t_m\}$ such that $OO^T=\sum s_it_i^2 I_n$ for some positive $(s_1,\ldots,s_m)\in \ZZ^m$ with $\sum s_i=n$. In a typical application, one would like to replace the indeterminates with matrix blocks. The problem is that blocks usually do not commute. Hence by substituting $t_i\gets B_i$, one is led to the condition that $B_iB_j^T=B_jB_i^T$ for all $i,j$. We say that $(B_i,B_j)$ is an \emph{Amicable Pair}.\\

Another prolific method is that of \emph{Cocyclic Matrices}. These are matrices $M$ indexed by a finite group $G$ such that for all $i,j$, $M_{i,j}=f(i^{-1}j)\omega(i^{-1},j)/\omega(i^{-1},i),$ where $f:G\to \{-1,0,1\}$ is some function, and $\omega:G\times G \to \{-1,1\}$ is a \emph{2-cocycle} (see below, this is slightly inconsistent with some of the literature, e.g. as in \cite{Horadam1995}). Allowing $f$ to take values in a ring, they form a *-Algebra (i.e. closed under addition, multiplication and transpose), which is why they are good candidates for being Hadamard matrices. Cocyclic matrices arose first from an abstract notion of orthogonality \cite{deLauney:Flannery:ADT:2011}. Finding Cocyclic Hadamard Matrices is an active area of research, see for example \cite{Alvarez:AAECC:2012}, \cite{Alvarez:JSC:2018}, \cite{OCathain:Roeder:DCC:2011}. Many known Hadamard Matrices are found to be Cocyclic, see for example \cite{Horadam:PUP:2007,Horadam:CC:2010}. There is a close connection between Cocyclic Matrices and automorphisms. If $M$ is a $\{-1,0,1\}$-matrix, an automorphism of $M$ is a pair $(L,R)$ of $\{\pm 1\}$ monomial matrices (= signed permutations) such that $LMR^T=M$. The collection of all automorphisms is a group, denoted $\Aut(M)$ under matrix multiplication. For a Cocyclic $M$, one has a central $2$-covering group of $G$ as an automorphism subgroup.\\

In this paper we will introduce a new construction for Hadamard and weighing matrices. We define

\begin{defn}
	Let $I,J$ be finite disjoint sets of (not necessarily commuting) indeterminates. A \emph{Formal Orthogonal Pair} (FOP) is a pair $(A,B)$ of rectangular matrices such that
	\begin{itemize}
		\item[(i)] The entries of $A$ are $\pm i$ for elements $i\in I\cup \{0\}$.
		\item[(ii)] The entries of $B$ are $\pm j$ for elements $j\in J\cup \{0\}$.
		\item[(iii)] $AB^T=0$.
	\end{itemize}
\end{defn}

The matrices $A$ and $B$ are assumed to be over the non-commutative ring $\ZZ\{I\cup J\}$ generated over $\ZZ$ by the symbols in $I\cup J$, where $-x$ is the negative of $x$. The equality $AB^T=0$ is to be understood over that ring.\\

Note that unlike with Orthogonal Designs, we do not assume that the indeterminates commute. In fact, by substituting arbitrary blocks, even rectangular (all of the same size) into an FOP, we still will get the orthogonality. Note that this notion of FOP is different from the notion of "Orthogonal Pairs" as appears in \cite{Craigen:ArsComb:1992}\\

In this paper, we will show how to construct orthogonal pairs, by techniques from Representation Theory and Group Cohomology.  By no means we can construct all of them. Those that we can construct come together with an automorphism group $G$, whose action preserves both $A$ and $B$. Moreover, the orthogonality property will be a consequence of the specific group action. Our construction may be explained in terms of Representation Theory (simpler), and Group Cohomology (harder, but more general). For the cohomology point of view, we refer the reader to \cite{goldberger2019cohomology} where the theory is explained in detail. As a side effect, we can also construct (formal) amicable pairs. 

\section{Projective Monomial Representations}\label{mon}
\subsection{Basic Theory}
We are given a finite group $G$. A \emph{Projective Representation} of $G$ is a map $\rho: G\to GL(V)$ where $V$ is a finite dimensional vector space over a field $F$, and for all $g,g'\ni G$, $\rho(gg')=\pm \rho(g)\rho(g')$. In our applications we will have $F=\QQ$. By abuse, we will say that $V$ is a Projective Representation of $G$. If we write $\rho(gg')=\omega(g,g') \rho(g)\rho(g')$ for $\omega:G\to \{\pm 1\}$, then $\omega$ is a $2$-cocycle:

$$ \omega(g',g'')\omega(g,g'g'')=\omega(gg',g'')\omega(g,g'), \ \forall g,g',g''\in G.$$ We say that $\omega$ is the cocycle of the projective representation $V$. A projective representation $V$ is \emph{irreducible} if there is no $G$-invariant subspace other than 0. It is a theorem that when $char F$ is prime to $2|G|$, then each representation breaks down to a direct sum of irreducible representations. A \emph{morphism} of representations $V\to V'$ is a homomorphism of spaces, that commutes with the $G$-action. Note that $V,V'$ need to have the same cocycle.\\

Attached to a 2-cocycle $\omega$, is a central extension $\widetilde G$ of $G$ such that 
$$1 \to \{\pm 1\}\to \widetilde{G} \to G \to 1$$ is exact, and $\{\pm 1\}$ is in the center of $\wG$. The 2-cocycle $\omega$ can be recovered from this extension by the existence of a functional section $s:G\to \wG$ such that $s(g)s(g')=\omega(g,g') s(gg')$. It is easy to see that every representation $\tilde \rho:\wG\to GL(V)$ of $\wG$ gives rise to a projective representation of $\rho:G\to GL(V)$ (with cocycle $\omega$), defined by $\rho=\tilde\rho\circ s$. Conversely, any projective representation of $G$ can be lifted to an ordinary representation of $\wG$. 
As with ordinary representations, we have\\
\begin{thm}[Schur's Lemma]
	If $T: V\to V'$ is a morphism between irreducible projective representations, then either (i) $T=0$, or (ii) $T$ is an isomorphism.$\Box$
\end{thm}

\begin{defn}
	A projective monomial representation is a projective representation $\rho:G\to GL(V)$ together with a basis $M=\{m_i\}$ of $V$ such that for any $g\in G$, $gm_i=\pm m_j$, $j=j(g,i)$. In other words, the representing matrix $[\rho(g)]_M$ in the basis $M$ is monomial (= a signed permutation).
\end{defn}
We call $M$ a \emph{monomial basis} of $V$.

\begin{ex}
	Let $X$ be a finite $G$-set. Then each $g\in G$ acts on $X$ as a permutation $\pi(g)$. Consider the vector space $V=F[X]$, of formal sums of elements of $X$ with coefficients in $F$. Then $V$ is a monomial (in fact permutation) representation of $G$, with monomial basis $M=X$.
\end{ex}

\begin{ex}
	Let $G=B_3\subset O(3)$ the symmetry group of the cube $C=[-1,1]^3\subset \RR^3$. Then $B_3$ is the group of all $\{\pm 1\}$-$3x3$ monomial matrices, and the inclusion $B_3\subset O(3)$ is a monomial representation $V$. Moreover,  $\hat V=\RR[C]$ is an 8-dimensional permutation representation of $G$, and it can be shown that $\hat V$ contains $V$ as a direct summand.
\end{ex}

\begin{ex}\label{cocrep} {\bf Cocyclic Representations.}\label{cocyclic} Let $G$ be a finite group and $\omega:G\times G \to \{\pm 1\}$ be a 2-cocycle. We define an action of $G$ on $F[G]$ by 
	$$(g,[g'])\mapsto \omega(g,g')[gg'].$$ Then it is easy to check that this defines a projective monomial representation of $G$.
\end{ex}

\begin{defn}
	Let $V,V'$ be projective monomial representations of $G$, with monomial bases $M,M'$. Let $Hom_M(V,V')$ be the set of all matrices $[T]_{M,M'}$ representing a morphism $T:V\to V'$ with respect to $M$ and $M'$. We also write $End_M(V)=Hom_M(V,V)$. 
\end{defn}
Given $V,M,V',M'$ as in the definition, we may define a $G$-action on the group $Hom(V,V')$ (=linear transformations $V\to V'$ as vector spaces) by setting
$$(gf)(v)\ := \ gf(g^{-1}v).$$ This defines an {\bf ordinary} $G$-representation on $Hom(V,V')$. By passing to matrix representations, it takes the form
$$ (g,[T]_{M,M'})\mapsto [g]_M [T]_{M,M'} [g]_{M'}^{-1}.$$ The elements of $Hom_M(V,V')$ are precisely those that are left invariant under this action.  

We claim:
\begin{prop}\label{prop24}
	If $V,V'$ are two projective monomial representations, then $Hom_M(V,V')$ is spanned over $F$ by a basis $\{E_i\}$ of $\{0,-1,1\}$-matrices with pairwise disjoint supports. Each $E_i$ is supported on an orbit of the $G$-action on $M\times M'$.
\end{prop}

\begin{proof}
	The elements of $H_M(V,V')$ are those that satisfy
	\be \label{inv} [T]_{M,M'}=[g]_M[T]_{M,M'}[g]_{M'}^{-1}.\ee By ignoring signs, the group $G$ acts by permutations on the monomial bases $M$ and $M'$ and hence on $M\times M'$, splitting $M\times M'=\sqcup_{i=1}^m O_i$ to a disjoint union of $G$-orbits.\\
	
	Pick a representative $(v,w) \in O_i$. Then the matrix value of $[T]_{M,M'}$  at each position $(v',w')\in O_i$ is determined uniquely from the value at $(v,w)$. More precisely, if the value at $(v,w)$ is $a$, and $gv=\eps v'$, $gw=\mu w'$, then the value at $(v',w')$ is $\eps\mu a$. By setting up $a=1$ (and $0$ at all other orbits) we get a $\{0,\pm 1\}$-matrix $E_i$, and every other matrix in $Hom_M(V,V')$ is a unique linear combination of the $E_i$.  			
\end{proof}

\begin{rem}It should be noted that 
	some of the $E_i$ may be zero. This happens when the $G$-action gives conflicting signs to an element (equivalently all elements) in $O_i$. We discuss this in \S\ref{ori} below. This phenomena is fundamental to our construction. 
\end{rem}

\subsection{Structure as *-algebras, modules and relation to Weighted Coherent Configrations}

For a monomial representation $V$, $End_M(V)$ is a matrix algebra: It is closed under addition and matrix multiplication. It is also a *-algebra (i.e. closed under the transpose), as is evident from equation \eqref{inv}. But it has more structure. It has a basis $\{E_i\}$ of $\{0,-1,1\}$-matrices with disjoint supports. This implies that $E_i^T=E_j$ for some $j$, and that \be E_iE_j=\sum_k \lambda^k_{i,j} E_k\ , \ \lambda^k_{i,j}\in \ZZ.\ee 

Moreover, by forgetting signs we have a permutation $G$-action on the monomial basis, and hence obtain a representation which we shall denote $|V|$. Then $End_M(|V|)$ is again a *-algebra, has a basis $\{F_i\}$ of $\{0,1\}$-matrices supported on the $G$-orbits of $M\times M$, and satisfy
\be F_iF_j=\sum_k \mu^k_{i,j} E_k\ , \ \mu^k_{i,j}\in \ZZ_{\ge 0}.\ee 
The *-algebra $End_M(|V|)$ is a special case of a \emph{Coherent Configuration} ~\cite{higman1975coherent}. Moreover, there exists a $\{0,-1,1\}$-matrix $W$, called a \emph{weight}-matrix ~\cite{higman1976coherent}, such that 
\be\label{W} End_M(|V|)\to End_M(V), \ \ A \mapsto A\circ W\ee
is a surjective map of vector spaces. Here $\circ$ is the Hadamard  ( = entrywise) product.\\

All of the above can be extended to modules. The group $Hom_M(V,V')$ is closed under multiplication by $End_M(V)$ on the left, and $End_M(V')$ on the right. Hence it is a \emph{bimodule}. Moreover, it has a $\{0,-1,1\}$-basis $P_i$ with disjoint support such that 
\be E_iF_j=\sum \theta^k_{i,j} F_k\ , \ \theta^k_{i,j}\in \ZZ.\ee
Likewise one can define the bimodule $Hom_M(|V|,|V'|)$ and a weight matrix with a property analogous to \eqref{W}.

\begin{ex} {\bf The algebra of Cocyclic Matrices}
	By taking the projective monomial representation $V$ of Example \ref{cocrep}, the algebra $End_M(V)$ is precisely the algebra of Cocyclic Matrices (cf. Introduction). 
\end{ex}

\subsection{Monomial and Induced Representations}\label{23} Suppose that $V$ is a projective monomial representation of $G$ with basis $M$. Then $M=\sqcup M_i$ is a union of $G$-orbits and accordingly $V=\bigoplus_i V_i$, a sum of projective monomial representations. For the remainder of section \ref{23}, assume $M$ is $G$-transitive. Pick up $m_0\in M$ and let $H=\{h\in G: hm_0=\pm m_0\}.$ Let $\widetilde{H}\subset \wG$ be the preimage of $H$. Then a $G$-sets $M\simeq G/H=\wG/\widetilde{H}$. Now, write $hm_0=\chi(h)m_0$ for $\chi:H\to \{\pm 1\}$. Then $m_0$ spans a 1-dimensional projective representation of $H$ and with respect to the cocycle $\omega|_{H\times H}$, which in turn is a coboundary (of $\chi$). 
\begin{thm}\label{lem1}
	Let $F[G]$ be the cocyclic projective representation. We define a projective right $H$ action by letting 
	$([g],h)\mapsto \omega(g,h)[gh].$ With this action 
	$$V\simeq F[G]\tensor_H (Fm_0),$$ as projective $G$ representations.
\end{thm}
\begin{proof}
	We give a map $V\to F[G]\tensor_{H} (Fm_0)$ as follows. For each $m\in M$, choose an element $g_m\in \wG$ s.t. $g_mm_0=\eps(m) m$, $\eps(g)=\pm 1$. Then we send $m\mapsto \eps(m)\omega(g_m,1)[g_m]\tensor m_0$. Acting by $g$ on the left, we must show that $gm$ maps to the action of $g$ on $\eps(g)\omega(g_m,1)[g_m]\tensor m_0$. We may write $gg_m=g_{n}h$ for the basis element $n$ such that $gm=\tau(g,m)n$, and $h\in H$.\\
	
	 Interpreting the $g$ action on $F[G]$, we have to check that
	 $\tau(g,m)n \mapsto \omega(g,g_m)\eps(g)\omega(g_m,1)[gg_m]\tensor m_0=\omega(g,g_m)\eps(m)\omega(g_m,1)[g_nh]\tensor m_0=\omega(g,g_m)\eps(m)\omega(g_m,1)\chi(h)\omega(g_n,h)[g_n]\tensor m_0.$ This amounts to showing that
	 \be\label{ugly} \tau(g,m)\omega(g_n,1)\eps(n)= \omega(g,m)\eps(m)\omega(g_m,1)\chi(h)\omega(g_n,h).\ee
	 
	 But by using $gg_m=g_nh$ and the projectivity of $V$, we  
	see that $\omega(g,g_m)^{-1}g(g_mm_0)=\omega(g_n,h)^{-1}g_n(hm_0)$. Then \eqref{ugly} follows using $hm_0=\chi(h)m_0$, $g_mm_0=\eps(m)m$, $g_nm_0=\eps(n)n$ and $gm=\tau(g,m)n$, and the fact that $\omega(x,1)$ is independent of $x$. This shows that $V\to F[G]\tensor_H Fm_0$ is a $G$-homomorphism. The elements $[g_m]\tensor m_0$ are an $F$-basis for $F[G]\tensor_H Fm_0$, hence our map is an isomorphism.
\end{proof}

For every 2-cocycle $\omega$ on $G$ and every subgroup $H\subset G$ such that $\omega|_{HxH}=d\chi$ is a coboundary, one can construct a 1-dimensonal projective representation $U=span\{u_0\}$ of $H$, by letting $hu_0:=\chi(h)u_0.$
\begin{defn}
	The projective induced representation is
	$$ Ind_H^G U = F[G]\tensor_H U.$$
\end{defn}

 The induced representation $V=Ind_H^G U$, is monomial, with monomial basis $[g_i]\tensor u_0$, $g_i$ being any choice of  right coset representatives: $G=\sqcup_i g_iH$.\\ 
 
 \begin{rem}
 	We see from here that \emph{Induction} produces all projective monomial representations, up to isomorphism. Note that the coboundary source $\chi$ \emph{is not unique}. It may be altered by a character: $\chi'(h)=\psi(h)\chi(h)$ where $\psi:H\to \{\pm 1\}$ is a homomorphism. Different choices of $\chi$ will result in \emph{non-isomorphic} induced representations.
 \end{rem}

\subsection{Formal Orthogonal Pairs}
We have seen that $Hom_M(V,V')$ has a basis of $\{0,-1,1\}$-matrices. It may sometimes happen that $Hom_M(V,V')=0$. This is the case, if and only if the irreducible components of $V$ and those of $V'$ have no isomorphic components in common (by Schur's Lemma). As a corollary of Proposition \ref{prop24}, we have:

\begin{lem}\label{orth}
	Suppose that $A\in Hom_M(W,V)$ and $B\in Hom_M(W,V')$. Then 
	$$Hom_M(V,V')=0 \ \ \implies \ \ AB^T=0.$$
\end{lem}

\begin{proof}
	We claim that $B^T\in Hom_M(V',W)$, from which it follows that $AB^T\in Hom_M(V',V)=0 \implies AB^T=0$. The elements of $Hom_M(V,W)$ are characterized as the matrices $C$ such that $[g]_M C [g]_{M'}^{-1}=C$ for all $g\in G$. Taking transpose shows that $B^T\in Hom_M(V',W)$.
\end{proof}

\begin{cor}
	Taking formal linear combinations of the $\{0,-1,1\}$ bases of $Hom_M(W,V)$ and $Hom_M(W,V')$, we get a \emph{Formal Orthogonal Pair}.
\end{cor}

\subsection{Orientability}\label{ori}
We are left with the question, how to construct a triple of projective monomial representations $(V,V',W)$, such that $Hom_M(W,V)\neq 0$, $Hom_M(W,V')\neq 0$, but $Hom_M(V',V)=0$? We can analyze the situation by decomposing $V,V',W$ into irreducible components over $F$. But this does not help us to construct the representations in the first place. Fortunately there is a group-theoretic practical test which can help in designing the representations.\\

First, choose a 2-cocyle $\omega:GxG\to \{\pm 1\}$, and two subgroups $H,H'\subset G$ such that the restriction to $H$ and $H'$ are couboundaries:
\begin{align}
\label{27}\omega(h_1,h_2)&=\chi(h_1h_2)\chi(h_1)^{-1}\chi(h_2)^{-1} \quad \forall h_1,h_2\in H\\
\label{28}\omega(h'_1,h'_2)&=\chi'(h'_1h'_2)\chi'(h'_1)^{-1}\chi'(h'_2)^{-1} \quad \forall h'_1,h'_2\in H.
\end{align}
Next we define two 1-dimensional projective representations $U,U'$ of $H,H'$, such that $(h,u_0)\mapsto \chi(h)u_0$ and $(h',u'_0)\mapsto \chi'(h')u'_0$. We induce to $G$ to obtain $V=Ind_H^G U$ and $V'=Ind_H^G U'$, with monomial bases $M=\{[g_i]\tensor u_0\}$ and $M'=\{[g'_i]\tensor u'_0\}$. The elements of $Hom_M(V,V')$ are the matrices $A$ s.t. $[g]_M A [g]_{M'}^{-1}=A$ for all $g\in G$. The action of $G$ partitions $A$ (more precisely $M\times M'=G/H\times G/H'$) into orbits. 
\begin{defn}
	An element $(m,m')\in M\times M'$ is \emph{orientable}, if for every element $t\in G$ such that $tm=\eps(t) m$ and $tm'=\eps'(t) m'$, we have $\eps(t)=\eps'(t)$. Otherwise it is \emph{non-orientable}.
\end{defn}

As a corollary of Proposition \ref{prop24} we have,
\begin{lem}
	Orientability depends only on the orbit. We have
	$$dim Hom_M(V,V')\ = \ \text{Number of orientable orbits.}$$
\end{lem}
Finally there is a technical criterion for orientability:
\begin{lem}[\cite{goldberger2019spectral}]\label{or}
	Let $V,V'$ be the induced projective monomial representations as above. Then a point $(m,m')=([g_i]\tensor u_0,[g'_j]\tensor [u'_0])$ is orientable, if and only if 
	$$ \chi(g_i^{-1}tg_i)=\chi'(g^{'-1}_jtg'_j)\quad \forall t\in g_iHg_i^{-1}\cap g'_jH'g^{'-1}_j.$$
\end{lem}
\begin{rem}
	Orientability depends on the choice of trivilalizations $\chi,\chi'$. One could modify $\chi$ to $\chi\psi$ where $\psi:H\to \{\pm 1\}$ is a homomorphism. 
\end{rem}

\section{The Cohomology Picture}
\subsection{A spectral sequence}
The theory in \S\ref{mon} can also be cast in terms of spectral sequences and cohomology. In fact this approach is more holistic, can be generalized to more contexts (action on coefficients, higher dimensional designs, connections to Brauer Groups, dual theories and more). Another advantage is that formulae can be derived from Homological Algebra. We only give here a brief statement, while a full treatment is developed in \cite{goldberger2019spectral}.\\

We begin with a finite group $G$, and two finite $G$-sets $X$ and $Y$. Let $\mu=\{\pm 1\}$, and $\mu^+=\{0,-1,1\}.$ An $X\times Y$ matrix is a $\mu^+$-valued matrix indexed by $X\times Y$. The group $G$ acts on those matrices by the rule
$$(gA)_{x,y}:=A_{g^{-1}x,g^{-1}y}.$$ Two matrices $A$ and $B$ are \emph{D-equivalent} if $A=D_1AD_2$ for $\mu$-valued diagonal matrices $D_i$. D-equivalence is denoted as $A\sim_D B$.
\begin{defn}
	A \emph{Cohomology Developed Matrix} (CDM) is an $X\times Y$ matrix such that 
	\be\label{CDM} gA\sim_D A \quad \forall g\in G.\ee
\end{defn}
It can be shown that CDMs are elements of $Hom_M(V,V')$ for projective monomial $V,V'$, and vice versa, all elements of 
$Hom_M(U,W)$ are CDM's.\\

Note that CDM's are closed under the Hadamard (= entrywise) product. Moreover, if we fix a set of orbits
$\mathcal O$ which is the support, then they form an Abelian group, denoted by $CDM(\mathcal O).$ Let $Mat(\O)$ be the group of all $\O$-supported matrices under the Hadamard product. We have
$$ CDM(\O)/D-\text{equiv.}=H^0(G,Mat(\O)/D-\text{equiv.}).$$
To compute the cohomology in the right hand side, we use a suitable resolution and obtain the following theorem. For simplicity we assume that the only trivial $D$-equivalence on $Mat(\O)$ is given by scalar matrices.

\begin{thm}
	There exists a first quadrant spectral sequence $E\implies H^*(G,Mat(\O)/D-\text{equiv.})$ whose $E^1$-page is given by
	\begin{align}
	    E^1_{0,q} &= H^q(G,\mu),\\
	    E^1_{1,q} &= \bigoplus_{H\in Stab(X)} H^q(H,\mu) \oplus \bigoplus_{H'\in Stab(Y)} H^q(H',\mu)\\
	    E^1_{2,q} &=\bigoplus_{Q\in Stab(\O)} H^q(Q,\mu)\\
	    E^1_{p,q} &=0 \quad  \forall p\ge 3. 
	\end{align}
	The differentials are given by the restriction maps. $Stab(X)$ is a the set of stabilizer subgroups of all points in $X$, up to conjugacy.
\end{thm} 

\begin{rem}
		The elements of $E^2_{0,2}$ are generated by 2-cocycles that are in the kernel of $d^1$ - the restriction to the stabilizers of points in $X$ and $Y$. This is manifested in conditions \eqref{27}-\eqref{28}. In addition, the elements of $E^3_{0,2}$ must have trivial image in $H^1(Q,\mu)$ for all $Q\in Stab(\O).$ This is the orientability condition.	
\end{rem}

\section{Examples}

\subsection{$G=$ A vector Space}
In this example we take $G=(\ZZ/2)^n$. The cohomology $H^2(G,\{\pm 1\})$ is represented by bilinear forms $\omega:G\times G\to \{\pm 1\},$ modulo isotropic forms (i.e $\omega(v,v)=1$). Take $n=4$, and the form $\omega(v,w)=\sum_{i \mod 4}v_iw_{i+1}.$ Consider the isotropic subspaces $H=span\{(1,1,1,1),(0,1,1,1)\}$, $H'=span\{(1,1,1,1),(0,1,0,1)\}$ and $K=span\{(0,0,0,1)\}$. Thus $\omega$ descends to the trivial class in the restriction to those spaces, with the following trivializations: 
\begin{align*}
\chi&=\chi_H((0,0,0,0))=1;&\ \chi_H(v)=-1 \text{ otherwise},\\
\chi'&=\chi_{H'}((0,1,0,1))=-1;&\ \chi_H(v)=1 \text{ otherwise},\\
\chi_K&=1.
\end{align*}
With this we obtain the FOP
{\small %
$$A \ = \ \left(\begin{array}{rrrrrrrr}
b & a & b & -a & a & b & -a & -b \\
b & a & -b & -a & -a & b & a & b \\
a & b & a & -b & b & a & -b & -a \\
a & b & -a & -b & -b & a & b & a
\end{array}\right),\ B \ = \ \left(\begin{array}{rrrrrrrr}
c & d & d & d & c & -c & c & d \\
-c & -d & d & -d & c & c & c & d \\
-d & -c & c & -c & d & d & d & c \\
-d & -c & -c & -c & -d & d & -d & -c
\end{array}\right).$$
}

\subsection{A partial weighing block-circulant family}
Let $n\ge 4$ be an integer, $G=\ZZ/2 \wr \ZZ/n$ be the wreath product of size $n2^n$. Then $V:=(\ZZ/2)^n \vartriangleleft G$. We work with $\omega=1$ the trivial cocycle. Next,we choose three subspaces $H_1,H_2,H_3\subset V$ and three characters $\chi_i:H_i\to \{\pm 1\}.$ We restrict our attention to spaces generated by standard vectors $\{e_i\}$. We encode the choice $(H,\psi)$ by a $\{0,-1,1\}$-vector $\rho$, where for instance $\rho=(1,-1,0,0,\ldots)$ means that $H$ is spanned by $e_1,e_2$ and $\chi(e_1)=1,\chi(e_2)=-1$. Write $X_i=G/H_i$. A straightforward application of Lemmas \ref{or} and \ref{orth} yields:

\begin{lem}\label{41} Let $(A,B)$ be the formal pair generated by the CDM's on $X_1\times X_3$ and $X_2\times X_3$. Then 
	 $AB^T=0$ if the circulant matrices $C_i$ whose first row is $\rho_i$ satisfy
		$(C_1C_2^T)_{i,j}<(|C_1||C_2|^T)_{i,j} \ \forall i,j.$ Furthermore,
		 $AA^T=\lambda I$ if $(C_1C_1^T)_{i,j}<(|C_1||C_1|^T)_{i,j} \ \forall i\neq j.$
\end{lem}

As an example, let $\rho_1=(-1,1,\ldots,1)$, $\rho_2=(1,-1,-1,1,1\ldots,1)$ and $\rho_3=(0,0,0,1,1,\ldots,1)$. These satsify conditions (a) and (b) in Lemma \ref{41}, hence we obtain an FOP $(A,B)$, each of size $n\times 8n$. $A$ consists of $3$ orientable orbits of row weight =8 each  and $B$ has $2$ orientable orbits. Moreover, the orientable orbits of $A$ have disjoint support from those of $B$. As a corollary we get:

\begin{thm}
	For every integer $n\ge 4$, there exist a partial weighing matrix $P$ of size $4n\times 8n$ and row weight 32. Moreover, there is a reordering of $P$ such that it is composed out of circulant $n\times n$ blocks.
\end{thm}

\begin{proof}(Sketch) We have $AA^T=24I_n$ and $BB^T=16I_n$, as formal matrices. It follows that the five orbits of $A$ and $B$ are mutually orthogonal. As they are disjoint, we can combine them into one formal orthogonal matrix $C=A+B$, s.t. $CC^T=32I_n$. Due to formality, we may substitute a $4x1$ column for each symbol, four of these taken from a Hadamard $H(4)$ and the fifth=0. This constructs $P$. The the $\ZZ/n$ factor in $G$ acts as automorphisms on $P$, and by analyzing orbits and taking care of signs we obtain the block-circulant structure.
\end{proof} 

\subsection{A Formal Amicable Pair}
It is easier to generate Amicable Pairs, i.e. a pair $(A,B)$ such that $AB^T$ is symmetric. For this we need $A$ and $B$ to be $X\times Y$ CDM's such that all $X\times X$ CDM's must be symmetric. The setup includes a group $G$ whose orbits on $X\times X$ are symmetric. An example is the Dihedral group $D_n$ realized as the set of affine transformations $G=\{x\mapsto ax+b\ | \ b\in \ZZ/n,\ a=\pm 1\}.$ Assume that $n$ is even. We choose $H=\{x\mapsto \pm x\}$ and $H'=\{id\}\cup \{x\mapsto x+n/2\}$. Take $\omega=1$ and let $\chi$ and $\chi'$ to be the unique nontrivial characters. The resulting Formal Amicable Pair is:

{\tiny
$$ A,B= \left(\begin{array}{rrrrrr}
c & -c & -a & a & -b & b \\
-b & -a & -b & c & c & a \\
-a & -b & c & -b & a & c \\
-c & c & a & -a & b & -b \\
-b & -a & -b & c & c & a \\
-a & -b & c & -b & a & c
\end{array}\right), \left(\begin{array}{rrrrrr}
f & -f & -d & d & -e & e \\
-e & -d & -e & f & f & d \\
-d & -e & f & -e & d & f \\
-f & f & d & -d & e & -e \\
-e & -d & -e & f & f & d \\
-d & -e & f & -e & d & f
\end{array}\right).$$
}

\subsection{Summary} Formal Orthogonal Pairs provide a new methodology to construct Hadamard and Weighing Matrices. The main tools in their construction involves representation theory and Group Cohomology. They can be used either via block substitutions, or by partitioning a matrix into formal orthogonal tuples.

\bibliographystyle{alpha}
\bibliography{biblfop}

\end{document}